\documentclass[11pt,reqno]{amsart}
\usepackage{amssymb,amsfonts}
\usepackage[all,arc]{xy}
\usepackage{enumerate}
\usepackage{mathrsfs}
\usepackage{geometry}
\usepackage{amsmath}
\usepackage{latexsym, bm}
\usepackage{indentfirst}
\usepackage{CJK}
\usepackage{amsthm}
\usepackage{hyperref}
\usepackage{cite}

\DeclareMathOperator{\tr}{tr}
\DeclareMathOperator{\re}{Re}
\DeclareMathOperator{\loc}{loc}
\DeclareMathOperator{\Ric}{Ric}
\DeclareMathOperator{\Rm}{Rm}
\newcommand{\p}{\frac{\partial}{\partial t}}
\newcommand{\sdd}{\sqrt{-1}\partial\bar{\partial}}

\newcommand{\om}{\omega}
\newcommand{\ra}{\rightarrow}
\newcommand{\vp}{\varphi}
\newcommand{\pa}{\partial}
\newcommand{\ep}{\epsilon}
\newcommand{\Om}{\Omega}
\newcommand{\ptd}{(\frac{\partial}{\partial t}-\Delta)}
\newcommand{\na}{\nabla}
\newcommand{\al}{\alpha}
\newcommand{\De}{\Delta}
\newcommand{\tl}{\tilde}

\newcommand{\la}{\lambda}
\newcommand{\lt}{\left}
\newcommand{\rt}{\right}

\newtheorem{thm}{Theorem}[section]

\newtheorem{lem}{Lemma}[section]
\newtheorem{prop}{Proposition}[section]

\theoremstyle{definition}

\numberwithin{equation}{section}

\begin{document}
\title{Weak Solutions of the Chern-Ricci flow on compact complex surfaces}
\author{Xiaolan Nie\vspace{-4ex}}
\maketitle
\begin{abstract}
In this note, we prove the existence of weak solutions of the Chern-Ricci flow through blow downs of exceptional curves, 
as well as backwards smooth convergence away from the exceptional curves on compact complex surfaces. The smoothing 
property for the Chern-Ricci flow is also obtained on compact Hermitian manifolds of dimension $n$ under a mild assumption.\end{abstract}

\section{Introduction}

 Let $(M, g_0)$ be a compact Hermitian manifold with associated $(1,1)$ form $\om_0$. The Chern-Ricci flow starting at $\omega_0$ is given by
\begin{align}\p \omega=-\Ric (\omega),\ \ \ \ \ \omega|_{t=0}=\omega_0,
\end{align}where $\Ric(\omega)=-\sqrt{-1}\partial\bar{\partial}\log \det g$ is the Chern-Ricci form of $\omega$. It was introduced by Gill~\cite{Gill} on 
manifolds with vanishing first Bott-Chern class and investigated by Tosatti and Weinkove~\cite{TW4} in details on general Hermitian manifolds.
 If the initial metric is K\"{a}hler, then it coincides with the K\"{a}hler-Ricci flow. 

Many nice properties of the flow have been found (see ~\cite{TW4,TW5, TWY, FTWZ}, etc.), some of which are analogous to those of the K\"{a}hler-Ricci flow.
 Let $$T=\sup\{t\geq 0|\exists\psi\in C^{\infty}(M),\,  \ \omega_0-t\Ric(\omega_0)+\sqrt{-1}\partial\bar{\partial }\psi>0\}.$$ It was proved in ~\cite{TW4} that
 there exists a unique maximal solution $\om(t)$ to the Chern-Ricci flow (1.1) on [0, T).
 It is expected that the Chern-Ricci flow is closely related to the geometry of the underlying manifold. In the case of $n=2$ (complex surfaces), the behavior
 of the flow is particularly interesting. 
Let $M$ be a compact complex surface with $\om_0$ a Gauduchon metric.  It was proved in~\cite{TW4} that the Chern-Ricci flow starting at $\om_0$ exists 
until either the volume of $M$ goes to zero, or the volume of a curve of negative self-intersection goes to zero. The Chern-Ricci flow is said
 to be \textit{collapsing} (\textit{non-collapsing}) at $T$ if the volume of $M$ with respect to $\om(t)$ goes to zero (stays positive) as $t\ra T^-$. \\

Suppose that the Chern-Ricci flow starting at $\om_0$ is non-collapsing at $T<\infty$.   It was shown in ~\cite{TW4} that $M$ contains finitely many 
disjoint  (-1)-curves $E_1, ... , E_k$ and thus there exists a map $\pi: M\rightarrow N$ onto a complex surface $N$ contracting each $E_i$ to a point $y_i\in N$.  
It was conjectured in ~\cite{TW5} that the flow blows down the exceptional curves and continues in a unique way on a new surface $N$. The conjecture requires
 smooth convergence of the metrics away from the (-1)-curves and global Gromov-Hausdorff convergence as $t\ra T^-$ and $t\ra T^+$. Denote 
$M'=M\setminus\cup_{i=1}^kE_i$. In ~\cite{TW5}, Tosatti and Weinkove prove the following theorem. 

\begin{thm}{\normalfont (Tosatti-Weinkove)} With the notation above, then the metrics $\om(t)$ converge to a smooth Gauduchon metric $\om_T$ on 
$M'$ in $C^{\infty}_{\loc}(M')$ as $t \ra T^-.$ Assume in addition
\begin{align}\omega_0-T\Ric(\omega_0)+\sqrt{-1}\partial\bar{\partial}f=\pi^*\om_N
\end{align}
for some $f\in C^{\infty}(M,\mathbb{R})$ and $\om_N$ a smooth (1,1) form  on $N$. Then there exists a distance function $d_T$ on $N$ such that $(N,d_T)$ is a
 compact metric space and $(M, g(t))\rightarrow (N,d_T)$ as $t\rightarrow T^-$ in the Gromov-Hausdorff sense.\end{thm}

Denote $\tilde{\om}_T=(\pi^{-1})^*\om_T$ the push-down of the limiting current $\om_T$ to $N$.\
To continue the flow on $N$, first we need to show that the Chern-Ricci flow on $N$ with singular initial metric $\tilde{\om}_T$ has a unique smooth solution on 
$(T, T']$ for some $T'>T$.  To prove the smooth convergence on compact subsets of $N'=N\setminus\{y_1, ... , y_k\}$, we need more precise 
estimates near the exceptional curves as $t\ra T^-$ and also estimates on $[T, T']\times N'$. We may assume $\om_N$ in condition (1.2) to be a 
Gauduchon metric after replacing $f$ by a new function (see Lemma 3.2 in ~\cite{TW5}) and denote
\[T_N=\sup\{t\geq T|\exists\psi\in C^{\infty}(M), \ \om_N-(t-T)\Ric(\omega_N)+\sqrt{-1}\partial\bar{\partial }\psi>0\}.\]

Using the construction of Song-Weinkove for the K\"{a}hler-Ricci flow in ~\cite{SW1}, we prove the following theorem.
\begin{thm} Assume that the condition (1.2) is satisfied. With the notation above, there exists a unique maximal solution $\om(t)$ to the
 equation: $$\om|_{t=T}=\tilde{\om}_T,\ \ \ \p\om=-\Ric(\om), \ \ for\  t\in(T, T_N),$$ which is smooth on $(T, T_N)$. Moreover, $\om(t)$
 converges to $\tilde{\om}_T$ in $C^{\infty}_{\loc}(N')$ as $t\rightarrow T^+.$
\end{thm}
When $\om_0$ is K\"{a}hler, the result is contained in the work of Song-Weinkove ~\cite{SW1}. We use the techniques
 of Song-Weinkove ~\cite{SW1}, Tosatti-Weinkove ~\cite{TW4,TW5} and a trick of Phong-Sturm ~\cite{PS} in the proof of the above theorem. \\

We actually obtain an existence theorem on $M$ of general dimension $n$. Let $\Omega$ be a smooth volume 
form on $M$ and $\hat{\om}_t=\om_0+t\chi$, where $\chi$ is a  closed (1,1) form locally defined by $\chi=\sdd\log\Omega$. 
Denote $PSH(M,\omega_0)$ the set of $\omega_0$-plurisubharmonic functions and let $$PSH_p(\omega_0, \Omega)=\{\varphi\in PSH(M,\omega_0)\cap L^{\infty}(M)|\frac{(\omega_0+\sqrt{-1}\partial\bar{\partial}\varphi)^n}{\Omega}\in L^p(M)\}.$$  Following the arguments of Song-Tian in \cite{ST}, we prove the smoothing property for the Chern-Ricci flow. \vspace{.62cm}

\begin{thm} Suppose that $\omega_0'=\omega_0+\sqrt{-1}\partial\bar{\partial}\varphi_0$ for some $\varphi_0\in PSH_p(\omega_0, \Omega), \ p>1$. 
Assume that $\om_0$ satisfies 

 \begin{align}\label{1.3}
\forall\  u\in PSH(M,\omega_0)\cap L^{\infty}(M),\ \ \  \int_M(\omega_0+\sqrt{-1}\partial\bar{\partial}u)^n=\int_M\omega_0^n,
 \end{align}

then there exists a unique family of smooth metrics $\om(t)$ on $(0, T)$ such that
\begin{itemize}
\item[(i)] $\p \om=-\Ric (\omega)$, for $t\in(0, T)$.
\item[(ii)] There exists $\vp\in C^0([0, T)\times M)\cap C^{\infty}((0, T)\times M)$ such that $\om=\hat{\om}_t+\sdd\vp$ and $\varphi(t)\rightarrow \varphi_0$ in $L^{\infty}(M)$ as $t\rightarrow 0^+$.
 \end{itemize}
    In particular, $\omega(t)\ra\omega_0'$ in the sense of currents as $t\rightarrow 0^+$.
\end{thm}
When $(M, \om_0)$ is a K\"{a}hler manifold, the result is contained in the work of Song and Tian ~\cite{ST} (see also ~\cite{CTZ}). When $M$ is a compact complex surface with a Gauduchon metric $\om_0$,  condition (\ref{1.3}) is satisfied and the above result follows. The proof of Theorem 1.3 is given in section 5.\\

\noindent\textbf{Acknowledgements.} This note is based on chapter 3 of the author's Ph.D. thesis. The author would like to
 thank her advisor Prof. Jiaping Wang for constant support, encouragement and many helpful discussions. The author also 
thanks Prof. Valentino Tosatti and Prof. Ben Weinkove for their encouragement and helpful discussions. In addition, the author is grateful to Haojie Chen for many helpful conversations.

After this note was completed, the author learned that related results were proved by Tat Dat T\^{o} in \cite{To}. The author would like to thank Tat Dat T\^{o} for sending his preprint.

\section{Construction of the weak solution}
In this section, we construct the solution in Theorem 1.2 explicitly. We will follow the construction of Song-Weinkove for the K\"{a}hler-Ricci flow ~\cite{SW1}. 

Without loss of generality, assume that $M$ contains only one exceptional curve $E$ for simplicity. Assuming that the condition (1.2) 
holds for a Gauduchon metric $\om_N$ on $N$. Define $$\hat{\om}_{t}=(1-\frac{t}{T})\om_0+\frac{t}{T}\pi^*\om_N,$$which are smooth
nonnegative forms on $[0, T]$. Let $\Omega=e^{f/T}\om_0^n$. (From now on, we will write n instead of 2 whenever our calculations 
hold for $n\geq 2$.)  If $\vp$ solves the parabolic complex Monge-Amp\`{e}re equation
\begin{align}\frac{\pa\vp}{\pa t}=\log\frac{(\hat{\om}_t+\sdd\vp)^n}{\Omega},\ \ \ \ \vp|_{t=0}=0\end{align} for $t<T$, then $\om=\hat{\om}_t+\sdd\vp$ 
solves the Chern-Ricci flow (1.1) on $[0, T)$.

By Lemma 3.3 of ~\cite{TW5}, there exists a uniform constant $C$ such that $|\vp|\leq C$ and $\dot{\vp}\leq C$, where we write $\dot{\vp}$
 for $\frac{\pa\vp}{\pa t}$. Then it follows that  as $t\ra T^-$, $\vp(t)$ converges pointwise on $M$ to a bounded function $\vp_T$ with $$\om_T=\hat{\om}_T+\sdd\vp_T\geq 0.$$
In particular, $\om(t)\ra \om_T$ in the sense of currents as $t\ra T^- $. From Lemma 5.1 of ~\cite{SW1} $\vp_T$ must be constant on $E$ 
since $\sdd\vp_T|_E=\om_T|_E\geq 0.$ Thus $\psi_T=(\pi^{-1})^*\vp_T$ is a bounded function on $N$ which is smooth on $N\setminus\{ y_0\}$ as $\pi$
 is the blow down map contracting $E$ to $y_0$. Then \begin{align}\tilde{\om}_T=\om_N+\sdd\psi_T\geq 0.
\end{align} 
\begin{lem}There exists $p>1$ such that $\tilde{\om}_T^n/\om_N^n\in L^p(N)$. Moreover, $\psi_T$ is continuous on $N$.
\end{lem}
 \begin{proof}
The argument of Lemma 5.2 in ~\cite{SW1} shows that $\tilde{\om}_T^n/\om_N^n\in L^p(N)$. The continuity of $\psi_T$ then
 follows from the results in ~\cite{KN} and Dinew-Ko{\l}odziej ~\cite{DK}.
\end{proof}

Given a smooth volume form $\Omega_N$, now we will construct a family of functions $\psi_{T, \epsilon}$ on $N$ which
 converge to $\psi_T$ in $L^{\infty}(N)$. For sufficiently small $\ep>0$ and $K$ large enough, define 
$$\Omega_\ep=(\pi|_{M\setminus E}^{-1})^*(\frac{|s|^{2K}_h\om^n(T-\epsilon)}{\ep+|s|^{2K}_h})+\ep\Omega_N \ \  \text{on}\ N\setminus\{y_0\}.$$
and $\Om_\ep|_{y_0}=\ep\Om_N|_{y_0}$. Here $s$ is a holomorphic section of the line bundle $[E]$ vanishing along the
 exceptional curve $E$ to order 1. Choose $h$ to be a smooth Hermitian metric on $[E]$ as in ~\cite{TW4} with curvature $R_h=-\sdd\log |s|_h^2$ such that for sufficiently small $\ep>0$,
$\pi^*\om_N-\ep R_h>0.$( see ~\cite{GH} for an argument of this.) Then the volume form $\Om_\ep\in C^k(N)$ for a fixed 
constant $k$ as $\pi$ is the blow down map and $K$ can be chosen to be sufficiently large. Moreover, $\Om_\ep$ converges
 to $\tilde{\om}_T^n$ in $C^{\infty}$ on compact subsets of $N\setminus\{y_0\}$ as $\ep$ goes to zero. By the result of Tosatti 
and Weinkove ~\cite{TW2}, there exist functions $\psi_{T,\ep}\in C^k(N)\cap C^{\infty}(N\setminus\{y_0\})$ such 
that $$(\om_N+\sdd\psi_{T,\ep})^n=C_\ep\Om_\ep$$ where the constants $C_\ep$ are chosen so that the integrals
 of both sides of the above equation match.  Write $F_\ep=C_\ep\Om_\ep/\Om_N$ and $F=\tilde{\om}_T^n/\Om_N$. 
Note that $C_\ep\ra 1$ as $\ep\ra 0$, then by the definition of $\Om_\ep$ we have $$\lim_{\ep\ra 0}||F_\ep-F||_{L^1(N)}=0.$$ By Lemma 2.1 and Kolodziej's stability result \cite{KN}, we have
\begin{align} \label{2.1}
\lim_{\ep\ra 0}||\psi_{T,\ep}- \psi_T||_{L^{\infty}(N)}=0.\end{align}
Let $\chi=\sdd\log\Omega_N$, then there exists $T'>T$ such that $$\hat{\om}_{t, N}=\om_N+(t-T)\chi$$ are smooth
 Gauduchon metrics for $t\in[T, T']$. For convenience, we will still write $\hat{\om}_t$ for $\hat{\om}_{t,N}$. Consider the parabolic complex Monge-Amp\`{e}re equations
$$\frac{\partial{\varphi_\ep}}{\partial{t}}=\log\frac{(\hat{\omega}_t+\sqrt{-1}\partial\bar{\partial}\varphi_\ep)^n}{\Omega_N},  \ \  \varphi_\ep|_{t=T}=\psi_{T,\ep}$$ on $[T, T']$. Then 
Theorem 1.3 and Lemma 2.1 give the following proposition.
\begin{prop}
There exists $\vp\in C^0(([T, T']\times N)\cap C^{\infty}((T, T']\times N)))$ such that \\
\item[]\ \ \ (i) $\vp_{\ep}\ra \vp $ in $L^{\infty}([T, T']\times N)$.\\
\item[]\ \ \ (ii) $\om(t)=\hat{\om}_t+\sdd \vp$ is the unique solution to the equation,\\[-.5cm]
\[\om|_{t=T}=\tl{\om}_T,\ \ \ \p\om=-\Ric(\om), \ \ for\  t\in(T, T']. \]
\end{prop}
To prove smooth convergence of $\om(t)$  on compact subsets of $N\setminus\{y_0\}$ as $T\ra T^+$, we need 
uniform estimates for $\om_{\ep}(t)=\hat{\om}_t+\sdd\vp_{\ep}$ on $[T, T']\times N\setminus \{y_0\}$, independent of $\ep$. 
To obtain these estimates, we need the bounds at time $T$, i.e. the bounds for $\tilde{\om}_{T, \ep}=\om_N+\sdd\psi_{T, \ep}.$ We will prove these in section 4.

\section{Higher order estimates as $t\ra T^-$}
In this section, we will prove a third order estimate and a bound for $|\Ric|$, which will be used to obtain the bounds for $\tilde{\om}_{T, \ep}$ in the next section.

As $t\ra T^-$, the smooth convergence of $\om(t)$ on compact subsets of $M\setminus\{E\}$ follows from Theorem 1.1 of ~\cite{TW5}. In 
particular, $C^{\infty}$ a priori estimates for $\om(t)$ on compact subsets away from the exceptional curves have already been obtained. 
However, we need more precise higher order global estimates on $M$ as $t\ra T^-$ to obtain the bounds for $\tilde{\om}_{T, \ep}$.

By Lemma 2.3 and 2.4 in ~\cite{TW5}, there exist positive constants $C$ and $K$ such that 
\begin{align} \label{e1}
\frac{|s|^{2K}_h}{C}\om_0\leq \om(t)\le \frac{C}{|s|^{2K}_h}\om_0.
\end{align}We may assume that $|s|_h^2\le 1$ on $M$ for convenience. By Lemma 2.1 of Guan-Li \cite{GL}, we can choose local coordinates around a point such that at this point 
\begin{align} \label{coor}
(g_0)_{i\bar{j}}=\delta_{ij},\  \ \ \pa_i(g_0)_{j\bar{j}}=0
\end{align}
for all $i, j$ and ($g_{i\bar{j}}$) is diagonal. Now choose such a coordinate system around a point. Following an argument in ~\cite{TW1}, we can get the inequality,
\begin{align} \label{e2}
\frac{|\nabla \tr_{\omega_0}\omega|_g^2}{\tr_{\omega_0}\omega} \leq\sum_{i,j,k} g^{i\bar{i}}g^{j\bar{j}}\partial_k g_{i\bar{j}}\partial_{\bar{k}}g_{j\bar{i}}+\sum_{i, j}g^{i\bar{i}}g^{i\bar{i}}|\partial_j (g_0)_{i\bar{j}}|^2.
\end{align} To see this, first applying the Cauchy-Schwarz inequality

\begin{align}
|\nabla \tr_{\om_0}\om|_g^2&=
\sum_i g^{i\bar{i}}(\sum_j\partial_ig_{j\bar{j}})(\sum_k\partial_{\bar{i}}g_{k\bar{k}})  \notag \\
&=\sum_ig^{i\bar{i}}|\sum_j\partial_ig_{j\bar{j}}|^2 \notag \\
&\leq \sum_ig^{i\bar{i}}(\sum_jg_{j\bar{j}})(\sum_jg^{j\bar{j}}|\partial_ig_{j\bar{j}}|^2) \notag \\ 
&=(\tr_{\om_0}\om)\left(\sum_{i,j}g^{i\bar{i}}g^{j\bar{j}}\partial_ig_{j\bar{j}}\partial_{\bar{i}}g_{j\bar{j}}\right).  \label{e3}
\end{align}

Note that $\om=\om_0+\theta(t)$ for a closed (1,1) form $\theta(t)$. 
Hence $$\partial_ig_{j\bar{j}}-\partial_jg_{i\bar{j}}=\partial_i(g_0)_{j\bar{j}}-\partial_j(g_0)_{i\bar{j}}.$$
Using (\ref{coor}), we get $$\partial_ig_{j\bar{j}}=\partial_jg_{i\bar{j}}-\partial_j(g_0)_{i\bar{j}}.$$ 
Similarly $\partial_{\bar{i}}g_{j\bar{j}}=\partial_{\bar{j}}g_{j\bar{i}}-\partial_{\bar{j}}(g_0)_{j\bar{i}}$. 
Then (\ref{e3}) gives\\
\begin{align}
\frac{|\nabla \tr_{\omega_0}\omega|^2}{\tr_{\omega_0}\omega}\leq\ 
&  \sum_{i,j} g^{i\bar{i}}g^{j\bar{j}}\partial_j g_{i\bar{j}}\partial_{\bar{j}}g_{j\bar{i}}-2\re \left (\sum_{i, j} g^{i\bar{i}}g^{j\bar{j}}\partial_j g_{i\bar{j}}\partial_{\bar{j}}(g_0)_{j\bar{i}}\right )\notag\\
\ & +\sum_{i,j} g^{i\bar{i}}g^{j\bar{j}}\partial_j (g_0)_{i\bar{j}}\partial_{\bar{j}}(g_0)_{j\bar{i}}\notag \\ =\ & \sum_{i,j} g^{i\bar{i}}g^{j\bar{j}}\partial_j g_{i\bar{j}}\partial_{\bar{j}}g_{j\bar{i}}-2\re \left (\sum_{i, j} g^{i\bar{i}}g^{j\bar{j}}\partial_i g_{j\bar{j}}\partial_{\bar{j}}(g_0)_{j\bar{i}}\right )\notag\\
\ & -\sum_{i,j} g^{i\bar{i}}g^{j\bar{j}}\partial_j (g_0)_{i\bar{j}}\partial_{\bar{j}}(g_0)_{j\bar{i}}\notag \\
\leq \ & \sum_{i,j} g^{i\bar{i}}g^{j\bar{j}}\partial_j g_{i\bar{j}}\partial_{\bar{j}}g_{j\bar{i}}-2\re \left (\sum_{i\neq j} g^{i\bar{i}}g^{j\bar{j}}\partial_i g_{j\bar{j}}\partial_{\bar{j}}(g_0)_{j\bar{i}}\right )\notag \\[10pt]
\leq \ & \sum_{i,j} g^{i\bar{i}}g^{j\bar{j}}\partial_j g_{i\bar{j}}\partial_{\bar{j}}g_{j\bar{i}}+\sum_{i\neq j} g^{j\bar{j}}g^{j\bar{j}}|\partial_i g_{j\bar{j}}|^2+\sum_{i\neq j}g^{i\bar{i}}g^{i\bar{i}}|\partial_j( g_0)_{i\bar{j}}|^2\notag \\[10pt]
 \leq \ &\sum_{i,j,k}g^{i\bar{i}}g^{j\bar{j}}\partial_k g_{i\bar{j}}\partial_{\bar{k}}g_{j\bar{i}}+\sum_{i, j}g^{i\bar{i}}g^{i\bar{i}}|\partial_j (g_0)_{i\bar{j}}|^2. \label{e4}
\end{align}
Thus we get the inequality (\ref{e2}). We will use it to prove our third order estimate. \\

Let ${(\Gamma_0)}_{ij}^k$ be the Christoffel symbols associated to $g_0$. It is convenient to compute using $\Phi_{ij}^{\ \ k}=\Gamma_{ij}^k-{(\Gamma_0)}_{ij}^k$ as in ~\cite{PSSW} (see also ~\cite{ShW}) . Denote $|\cdot|$ the norm with respect to the metric $g$. Consider $$S= |\nabla_{g_0}g|^2=g^{i\bar{p}}g^{j\bar{q}}g_{k\bar{r}}\Phi_{ij}^{\ \ k}\Phi_{\bar{p}\bar{q}}^{\ \ \bar{r}}.$$

\begin{prop}
There exist positive constants $\lambda$ and $C$ such that for $t\in [0, T)$, 
$$S\leq \frac{C}{|s|_h^{2\lambda}}.$$
\end{prop}
\begin{proof}
Let $$H=\frac{S}{(|s|_h^{-\alpha}-\tr_{\omega_0}\omega)^2}+|s|_h^{\beta}\tr_{\omega_0}\omega -At,$$ where $\alpha$ and $\beta$ are constants to be determined and at least large enough such that
$$\frac{1}{2}|s|_h^{-\alpha}<|s|_h^{-\alpha}-\tr_{\omega_0}\omega<|s|_h^{-\alpha}$$ and $$|\nabla |s|_h^{\beta}|\leq C |s|_h^{\frac{3}{4}\beta},\ \ \ \ |\Delta |s|_h^{\beta}|\leq C |s|_h^{\frac{\beta}{2}}.$$
Computing the evolution of $H$, 
 \begingroup
\addtolength{\jot}{.8em}
\begin{align}
 \label{e5}
\begin{split}(\frac{\partial}{\partial t}-\Delta)H
=&\frac{1}{(|s|_h^{-\alpha}-\tr_{\omega_0}\omega)^2}(\frac{\partial}{\partial t}-\Delta)S\\
&-\frac{2S}{(|s|_h^{-\alpha}-\tr_{\omega_0}\omega)^3}(\frac{\partial}{\partial t}-\Delta)(|s|_h^{-\alpha}-\tr_{\omega_0}\omega)
\\&+\frac{4\re \nabla S \cdot \overline{\nabla} (|s|_h^{-\alpha}-\tr_{\omega_0}\omega)}{(|s|_h^{-\alpha}-\tr_{\omega_0}\omega)^3}-\frac{ 6S |\nabla (|s|_h^{-\alpha}-\tr_{\omega_0}\omega)|^2}{(|s|_h^{-\alpha}-\tr_{\omega_0}\omega)^4}-A\\
&+(\frac{\partial}{\partial t}-\Delta)(|s|_h^{\beta}\tr_{\omega_0}\omega).
\end{split}\end{align}
\endgroup
From \cite{ShW} and (\ref{e1}), we have the estimates
\begin{align}  \label{e6}
(\frac{\partial}{\partial t}-\Delta)S\leq C|s|_h^{-\alpha_1}(1+S^{3/2})-\frac{1}{2}|\overline{\nabla}\Phi|^2,
\end{align}
and \begin{align}  \label{e7}
(\frac{\partial}{\partial t}-\Delta)\tr_{\omega_0}\omega\leq -\frac{|s|_h^{\alpha_2}S}{C}+C|s|_h^{-\alpha_2}-\frac{1}{2}\sum_{i,j,k}g^{j\bar{j}}g^{i\bar{i}}\partial_k g_{i\bar{j}}\partial_{\bar{k}}g_{j\bar{i}}.
\end{align}
Also,  by (\ref{e2})
$$\frac{|\nabla \tr_{\omega_0}\omega|^2}{\tr_{\omega_0}\omega} \leq \sum_{i,j,k}g^{j\bar{j}}g^{i\bar{i}}\partial_k g_{i\bar{j}}\partial_{\bar{k}}g_{j\bar{i}}+C|s|_h^{-\alpha_3}.$$
Compute
\begin{align*}
&\ptd(|s|_h^{\beta}\tr_{\omega_0}\omega) \\[6pt]
&=|s|_h^{\beta}\ptd\tr_{\omega_0}\omega-(\tr_{\omega_0}\omega)\Delta |s|_h^{\beta}-2\re(\nabla |s|_h^{\beta}\cdot \overline{\nabla}\tr_{\omega_0}\omega) \\[6pt]
&\leq -\frac{1}{C'}|s|_h^{2\beta}S+C'.
\end{align*}
In the last inequality we use
 \begingroup
\addtolength{\jot}{.7em}
\begin{align*}
2|\re(\nabla |s|_h^{\beta}\cdot \overline{\nabla}\tr_{\omega_0}\omega)|&\leq C+\frac{1}{C}|\nabla |s|_h^{\beta}|^2|\nabla \tr_{\omega_0}\omega |^2\\
&\leq C+\frac{1}{C}|\nabla |s|_h^{\beta}|^2(\tr_{\omega_0}\omega)(\sum_{i,j,k}g^{i\bar{i}}g^{j\bar{j}}\partial_k g_{i\bar{j}}\partial_{\bar{k}}g_{j\bar{i}}+C|s|_h^{-\alpha_3})\\
&\leq C+\frac{1}{2}|s|_h^{\beta}\sum_{i,j,k}g^{i\bar{i}}g^{j\bar{j}}\partial_k g_{i\bar{j}}\partial_{\bar{k}}g_{j\bar{i}}.
\end{align*} for $\beta$ large enough. Then fix $\beta$.
\endgroup
Together with (\ref{e6}) and (\ref{e7})we get
 \begingroup
\addtolength{\jot}{.7em}
\begin{align*}
\ptd H&\leq C|s|_h^{2\alpha-\alpha_1}(1+S^{3/2})-\frac{1}{2}|s|_h^{2\alpha}|\overline{\nabla}\Phi|^2\\
&+C_0|s|_h^{3\alpha}|\Delta|s|_h^{-\alpha}|S+C_0|s|_h^{3\alpha-\alpha_2}S-|s|_h^{3\alpha+\alpha_2}\frac{S^2}{C_0}\\
&+4\re \frac{\nabla S\cdot \overline{\nabla}(|s|_h^{-\alpha}-\tr_{\omega_0}\omega)}{(|s|_h^{-\alpha}-\tr_{\omega_0}\omega)^3}+(-\frac{1}{C'}|s|_h^{2\beta}S+C')-A.
\end{align*}
\endgroup
As $$|\overline{\nabla}S|\leq 2S^{1/2}|\overline{\nabla}\Phi|,\ \ \ |\nabla \tr_{\omega_0}\omega|\leq C |s|_h^{-\alpha_4}S^{1/2}$$ and $|\nabla |s|_h^{-\alpha}|\leq C|s|_h^{-\alpha-\alpha_5},$ we have
\begin{align*}
\left |4\re \frac{\nabla S\cdot \overline{\nabla}(|s|_h^{-\alpha}-\tr_{\omega_0}\omega)}{(|s|_h^{-\alpha}-\tr_{\omega_0}\omega)^3} \right |&\leq CS^{1/2}|\overline{\nabla}\Phi|(|s|_h^{2\alpha-\alpha_5}+|s|_h^{3\alpha-\alpha_4}S^{1/2})\\[4pt]
&\leq \frac{1}{2}|s|_h^{2\alpha}|\overline{\nabla}\Phi|^2+C_1(|s|_h^{2\alpha-2\alpha_5}S+|s|_h^{4\alpha-2\alpha_4}S^{2}).
\end{align*}
Also $$C|s|_h^{2\alpha-\alpha_1}S^{3/2}\leq |s|_h^{3\alpha+\alpha_2}\frac{S^{2}}{2C_0}+C_1|s|_h^{\alpha-2\alpha_1-\alpha_2}S.$$
By choosing $\alpha$ and then $A$ sufficiently large, we have
$$\ptd H\leq 0.$$
Thus $H$ has a uniform upper bound and we obtain the desired estimate for $S$.
\end{proof}

\begin{prop}
There exist positive constants $\lambda$ and $C$ such that for $t\in [0, T)$, 
$$|\Ric|\leq \frac{C}{|s|_h^{2\lambda}}.$$
\end{prop}
\begin{proof}
First, we have the evolution equation
\begin{align*}
(\frac{\partial}{\partial t}-\Delta ) R_{j\bar{k}}=&\nabla_{\bar{k}}T^{\ r}_{lj}R^{\ l}_r+T^{\ r}_{lj}\nabla_{\bar{k}}R^{\ l}_r+R^{\ \ \ r}_{l\bar{k}j}R^{\ l}_r\\
&-R^{\ \ \bar{s}l}_{l\bar{k}}R_{j\bar{s}}+\nabla^{\bar{q}}T_{\bar{k}\bar{q}}^{\ \ \bar{s}}R_{j\bar{s}}+T_{\bar{k}\bar{q}}^{\ \ \bar{s}}\nabla^{\bar{q}}R_{j\bar{s}}
\end{align*}
Then it follows from Lemma 3.4 in \cite{N} and Proposition 3.1 that \begin{align}  \label{e8}
(\frac{\partial}{\partial t}-\Delta ) |\Ric|&=\frac{1}{2|\Ric|}(\frac{\partial}{\partial t}-\Delta ) |\Ric|^2+2|\nabla |\Ric||^2)\notag \\[3pt]
&\leq |s|_h^{-\al}(|\nabla \Ric|+|\Rm|^2+1)-\frac{|\nabla \Ric|^2}{|\Ric|}+\frac{|\nabla |\Ric||^2}{|\Ric|}.
\end{align}
for some constant $\al>0$.
Consider $$H= |s|_h^{3\al}|\Ric|+ |s|_h^{4\beta}S-At,$$ where $\al$ and $\beta$ are constants to be determined and are at least large enough such that 
\begin{align} \label{e9}
|\overline{\na}|s|_h^{3\al}|\leq C|s|_h^{2\al}, \ \ \ \ \  |\Delta|s|_h^{3\al}|\leq C|s|_h^{2\al}\end{align}and 
\begin{align} \label{e10}
|\overline{\na}|s|_h^{4\beta}|\leq C|s|_h^{3\beta}, \ \ \   \ \ |\Delta|s|_h^{4\beta}|\leq C|s|_h^{3\beta}.\end{align}
Assume that $H$ achieves maximum at a point $(t_0, z_0),\ t_0>0$ and $|\Ric|>1$ at $(t_0, \ z_0)$.
By (\ref{e6}) and the estimate for $S$,\begin{align} \label{e11}
(\frac{\partial}{\partial t}-\Delta)S\leq C|s|_h^{-\beta}-\frac{1}{2}|\overline{\na}\Phi|^2,
\end{align} for sufficiently large $\al$.
Together with (\ref{e8}),(\ref{e9}) and (\ref{e10}), we obtain
\begingroup
\addtolength{\jot}{.8em}
\begin{align}  \label{e12}
(\frac{\partial}{\partial t}-\Delta )H\leq & |s|_h^{2\alpha} (|\nabla \Ric|+|\Rm|^2)+|s|_h^{3\alpha}\left(\frac{|\nabla |\Ric||^2}{|\Ric|}-\frac{|\nabla \Ric|^2}{|\Ric|}\right)\nonumber\\
&-2\re\left(\na|\Ric|\cdot \overline{\na}|s|_h^{3\alpha}\right)-\Delta |s|_h^{3\alpha}|\Ric|-2\re\left(\na S\cdot \overline{\na}|s|_h^{4\beta}\right)\notag\\
&+C|s|_h^{3\beta}-\frac{1}{2}|s|_h^{4\beta}|\overline{\na}\Phi|^2-\De|s|_h^{4\beta} S-A.
\end{align}
\endgroup
As $\na H=0$ at $(t_0, z_0)$, we have 
\begin{align*}
|s|_h^{3\alpha}\overline{\na} |\Ric|=-\overline{\na}|s|_h^{3\alpha}|\Ric|-|s|_h^{4\beta}\overline{\na}S-\overline{\na}|s|_h^{4\beta}S
\end{align*}
Combining with (\ref{e9}) and (\ref{e10}) and $|\nabla |\Ric||\leq |\nabla\Ric|$, we get
\begingroup
\addtolength{\jot}{.8em}
\begin{align}  \label{e13}
\frac{|\nabla |\Ric||^2}{|s|_h^{-3\alpha}|\Ric|}&\leq C|s|_h^{2\al}|\na{\Ric}|+\frac{|\overline{\na}S||\na{\Ric}|}{|s|_h^{-4\beta}|\Ric|}+C\frac{S|\na{\Ric}|}{|s|_h^{-3\beta}|\Ric|}\notag\\
&\leq \frac{|\nabla \Ric|^2}{4|s|_h^{-3\alpha}|\Ric|}+C_1|s|_h^{\al}|\Ric|+C_2\frac{|s|_h^{7\beta}|\overline{\na}\Phi|^2}{|s|_h^{3\al}|\Ric|}\notag \\
&\ \ \ \ +C_3\frac{|s|_h^{4\beta}}{|s|_h^{3\alpha}|\Ric|},
\end{align}
\endgroup
where we use $|\overline{\na} S|\leq 2S|\overline{\na}\Phi|^2$ and $S\leq C|s|_h^{-\beta}$ in the second inequality. 
Note that $$|\Rm|^2\leq\frac{3}{2} |\overline{\nabla}\Phi|^2+|s|_h^{-2\alpha}$$ for $\alpha$ large as $$\nabla_{\bar{q}}\Phi_{ij}^{\ \ k}=-R_{i\bar{q}j}^{\ \ \  k}+(R_0)^{\ \ \ k}_{i\bar{q}j}.$$Then we have
\begingroup
\addtolength{\jot}{.8em}
\begin{align*}
\ptd H\leq&  C|s|_h^{\al}|\Ric|+C\frac{|s|_h^{7\beta}|\overline{\na}\Phi|^2}{|s|_h^{3\al}|\Ric|}+C\frac{|s|_h^{4\beta}}{|s|_h^{3\alpha}|\Ric|}\\
&+C|s|_h^{2\beta}|\overline{\na}\Phi|-\frac{1}{4}|s|_h^{4\beta}|\overline{\na}\Phi|^2+C-A.
\end{align*}
\endgroup
Assume at $(t_o, z_0)$, $|s|_h^{3\al}|\Ric|\geq 1$ (otherwise, $H\leq 2$ and the bound for $|\Ric|$ follows). By (\ref{e1}),  $$|\Ric|\leq|\overline{\na}\Phi| +|s|_h^{-\al_1}.$$
For $\al$ large enough, the term $|s|_h^{\al}|\Ric|$ can be controlled by $\frac{1}{6}|s|_h^{4\beta}|\overline{\na}\Phi|^2$. 
Choosing $\al$ and then $A$ sufficiently large, we get $$\ptd H\leq 0,$$ therefore we have the uniform upper bound for $H$ and the proposition follows.
\end{proof}

\section{Smooth convergence as $t\ra T^+$}

In this section, we will prove sharper bounds on $\vp_{\ep}$ and $\vp$ and then obtain the smooth convergence of the metrics $\om(t)$
 to $\tilde{\om}_T$ on compact subsets of $N\setminus\{y_0\}$ as $t\ra T^+$.

Recall that $\tl{\om}_{T,\ep}=\om_N+\sdd \psi_{T,\ep}.$ For simplicity, write $\tl{\om}=\tl{\om}_{T,\ep}$, $\hat{\om}=\hat{\om}_T=\om_N$ and
 $|s|^2_h$ for $(\pi|^{-1}_{M\setminus E})^*(|s|^2_h)$ in the proof of the following lemma. Then
$$\tl{\om}^n=(\hat{\om}+\sdd \psi_{T,\ep})^n=e^{F_{\ep}}\hat{\om}^n,\ \ \ \text{where}\ \  F_{\ep}=\log\lt( \frac{C_{\ep}\Om_{\ep}}{\hat{\om}^n}\rt).$$
\begin{lem}
There exist constants $\la>0$ and $C>0$, independent of $\ep$, such that
$$\frac{|s|_h^{2\la}}{C}\om_N\leq\tl{\om}_{T, \ep}\leq\frac{C}{|s|_h^{2\la}}\om_N$$ on $N\setminus \{y_0\}.$\end{lem}

\begin{proof}As $\psi_{T,\ep}$ is uniformly bounded by (\ref{2.1}), there exists a constant $C_0$ such that $\psi_{T,\ep}+C_0\geq 1$. Take $\ep_0$ 
small enough such that $\hat{\om}-\ep_0R_h\ge c\hat{\om}$ for some positive constant $c$, where $R_h=-\sdd\log|s|_h^2$ is the curvature of
 the Hermitian metric $h$. Let $\tl{\psi}_{T, \ep}=\psi_{T,\ep}-\ep_0\log|s|_h^2$ and define $$H=\log\tr_{\hat{\om}}\tl{\om}-A\tl{\psi}_{T,\ep}+\frac{1}{\tl{\psi}_{T,\ep}+C_0}.$$ 
Note that $H(y)$ goes to negative infinity as $y$ tends to $y_0$. Compute at a point in $N\setminus\{y_0\}$. Assume $\tr_{\hat{\om}}\tl{\om}\ge 1$ at this point.
From section 9 of ~\cite{TW4}, we have 
\begin{align} \label{4.1}
\De_{\tl{\om}}\log \tr_{\hat{\om}}\tl{\om}\ge \frac{2}{(\tr_{\hat{\om}}\tl{\om})^2}\re\lt(\tl{g}^{k\bar{q}}\hat{T}^p_{pk}\pa_{\bar{q}} \tr_{\hat{\om}}\tl{\om}\rt)-C\tr_{\tl{\om}}\hat{\om}-|\De_{\hat{\om}}F_{\ep}|-C.
\end{align}Assume that $H$ achieves a maximum at $z_0$. As $\pa_{\bar{q}} H=0$ at $z_0$, that is
$$\frac{\pa_{\bar{q}} \tr_{\hat{\om}}\tl{\om}}{\tr_{\hat{\om}}\tl{\om}}-A\pa_{\bar{q}}\tl{\psi}_{T,\ep}-\frac{\pa_{\bar{q}} \tl{\psi}_{T,\ep}}{(\tilde{\psi}_{T,\ep}+C_0)^2}=0,$$
we get 
\begin{align}
&\left |\frac{2}{(\tr_{\hat{\om}}\tl{\om})^2}\re(\tl{g}^{k\bar{q}}\hat{T}^j_{jk}\hat{\na}_{\bar{q}} \tr_{\hat{\om}}\tl{\om})\right |\notag\\[11pt]
=\ &\left |\frac{2}{\tr_{\hat{\om}}\tl{\om}}\re(\tl{g}^{k\bar{q}}\hat{T}^j_{jk}(A+\frac{1}{(\tl{\psi}_{T,\ep}+C_0)^2})\pa_{\bar{q}} \tl{\psi}_{T,\ep} \right |\notag\\[11pt]
\leq\ & \frac{|\pa \tl{\psi}_{T,\ep}|_{\tl{\om}}^2}{(\tl{\psi}_{T,\ep}+C_0)^3}+\frac{(\tl{\psi}_{T,\ep}+C_0)^3CA^2\tr_{\tl{\om}}\hat{\om}}{(\tr_{\hat{\om}}\tl{\om})^2}. \label{4.2}
\end{align}
If at $ z_0$, $(\tr_{\hat{\om}}\tl{\om})^2\le A^2(\tl{\psi}_{T,\ep}+C_0)^3$, then 
$$H\le\log A+\frac{3}{2}\log(\tl{\psi}_{T,\ep}+C_0)-A\tl{\psi}_{T,\ep}+\frac{1}{\tl{\psi}_{T,\ep}+C_0}.$$
As $\tl{\psi}_{T,\ep}+C_0\geq 1$,  we have an upper bound for $H$ and thus $\tl{\om}$ is bounded from above.
Otherwise, $A^2(\tl{\psi}_{T,\ep}+C_0)^3\le(\tr_{\hat{\om}}\tl{\om})^2$ at the maximum point. Moreover, by the definition of $\Om_{\ep}$, it follows from (\ref{e1}) and Proposition 3.2 that
 $$|\De_{\hat{\om}}F_{\ep}|\leq \frac{C}{|s|^{2\beta}}$$ for uniform constants $C$ and $\beta$. Together with (\ref{4.1}) and (\ref{4.2}) we have\\
\begingroup
\addtolength{\jot}{.8em}
\begin{align}
\De_{\tl{\om}}H=\ &\De_{\tl{\om}}\log \tr_{\hat{\om}}\tl{\om}+\frac{2|\pa \tl{\psi}_{T,\ep}|^2_{\tl{\om}}}{(\tl{\psi}_{T,\ep}+C_0)^3}\notag\\
\ &-\lt(A+\frac{1}{(\tl{\psi}_{T,\ep}+C_0)^2}\rt)\tr_{\tl{\om}}(\tl{\om}-\hat{\om}+\ep_0R_h)\notag\\
\geq\ &\De_{\tl{\om}}\log \tr_{\hat{\om}}\tl{\om}+cA\tr_{\tl{\om}}\hat{\om}+\frac{2|\pa \tl{\psi}_{T,\ep}|^2_{\tl{\om}}}{(\tl{\psi}_{T,\ep}+C_0)^3}-(A+1)n\notag\\
\geq \ &cA\tr_{\tl{\om}}\hat{\om}-\frac{(\tl{\psi}_{T,\ep}+C_0)^3CA^2\tr_{\tl{\om}}\hat{\om}}{(\tr_{\hat{\om}}\tl{\om})^2}-\frac{C}{|s|_h^{2\lambda}}\notag\\
\geq \ &(cA-C)\tr_{\tl{\om}}\hat{\om}-\frac{C}{|s|_h^{2\lambda}}. \label{4.3}
\end{align}
\endgroup
At the maximum point, $\De_{\tl{\om}}H\le 0,$ therefore $$\tr_{\tl{\om}}\hat{\om}\leq \frac{C}{|s|_h^{2\lambda}}$$ for sufficiently large $A$.
Then at $z_0$,
\begin{align*}\tr_{\hat{\om}}\tl{\om}\le\frac{1}{(n-1)!}(\tr_{\tl{\om}}\hat{\om})^{n-1}\frac{\tl{\om}^n}{\hat{\om}^n}\le \frac{C}{|s|_h^{2\beta}}
\end{align*}
as $\tl{\om}^n=C_{\ep}\Om_{\ep}$ has an upper bound by definition of $\Om_{\ep}$ and (\ref{e1}) holds. Thus there exists $C>0$, independent of $\ep$, 
such that $H\leq C$ for sufficiently large $A$. Since $\tl{\psi}_{T, \ep}+C_0\geq 1$, we obtain the desired estimates for $\tl{\om}_{T, \epsilon}=\tl{\om}$. \end{proof}

Recall that
\begin{align}\label{4.4}
\om_{\ep}(t)=\hat{\om}_t+\sdd\vp_{\ep}
\end{align} 
for $t\in[T, T'].$ From Lemma 5.4 of ~\cite{SW1}, we have the following volume bound. 
 \begin{lem}
There exist constants $\la>0$ and $C>0$, independent of $\ep$, such that 
$$\frac{\om_{\ep}^n}{\Om_N}\leq \frac{C}{|s|_h^{2\la}}$$
on $[T, T']\times (N\setminus \{y_0\})$.
\end{lem}
With Lemma 4.1 and Lemma 4.2, the upper bound for $\om(t)$ can be obtained by using the argument of Tosatti-Weinkove ~\cite{TW5} 
(see also ~\cite{PS}). For simplicity, we write 
$\hat{\om}=\om_N,\ \ \tl{\om}=\om_{\ep}$ in the proof of the following lemma.
\begin{lem}There exist constants $\la>0$ and $C>0$, independent of $\ep$, such that on $[T, T']\times (N\setminus \{y_0\}),$
$$\frac{|s|_h^{2\la}}{C}\om_N\leq\om_{\ep}\leq\frac{C}{|s|_h^{2\la}}\om_N.$$
\end{lem}
\begin{proof}Take $\ep_0$ small enough such that $\hat{\om}_t-\ep_0 R_h\ge c\hat{\om}$ for any $t\in [T, T']$ for some constant $c>0$. 
Let $\tl{\vp}_{\ep}=\vp_{\ep}-\ep_0\log|s|_h^2$. By Proposition 2.1, there exists a positive constant $C_0$, such that  $\tl{\vp}_{\ep}+C_0\geq 1.$
Define $$H=\log\tr_{\hat{\om}}\tl{\om}-A\tl{\vp}_{\ep}+\frac{1}{\tl{\vp}_{\ep}+C_0},$$ where $A$ is a positive constant to be determined. We have
$$H|_{t=T}\leq|s|_h^{2A\ep_0}\log\tr_{\om_N}\tl{\om}_{T, \ep}-A\psi_{T,\ep}+1.$$ By Lemma 4.1, $H$ is uniformly bounded from above 
at time $T$. Moreover, $H(t, y)$ tends to negative infinity as $y$ tends to $y_0$, for any $t\in[T, T']$. Compute at a point in ${N\setminus\{y_0\}}$ with $\tr_{\hat{\om}}\tl{\om}\geq 1$. 
From Proposition 3.1 in \cite{TW4}, 
\begin{align}\label{4.5}
\lt(\p-\De_{\tl{\om}}\rt)\log\tr_{\hat{\om}}\tl{\om}  
\leq\frac{2}{(\tr_{\hat{\om}}\tl{\om})^2}\re\left(\tl{g}^{k\bar{q}}\hat{T}_{kp}^p\pa_{\bar{q}}\tr_{\hat{\om}}\tl{\om}\right)+C\tr_{\tl{\om}}\hat{\om}.
\end{align}

Assume $H$ achieves its maximum at $(t_0, z_0)$, we have $\pa_{\bar{q}}\tr_{\hat{\om}}\tl{\om}=0$ at this point, thus 
\begin{align}
&\left|\frac{2}{(\tr_{\hat{\om}}\tl{\om})^2}\re\left(\tl{g}^{k\bar{q}}\hat{T}_{kp}^p\pa_{\bar{q}}\tr_{\hat{\om}}\tl{\om}\right)\right|\notag\\[8pt]
 \leq\  &\lt |\frac{2}{\tr_{\hat{\om}}\tl{\om}}\re\lt(\tl{g}^{k\bar{q}}\hat{T}_{kp}^p\lt(A+\frac{1}{(\tl{\vp}_{\ep}+C_0)^2}\rt)\pa_{\bar{q}}\tl{\vp}_{\ep}\rt)\rt|\notag\\[8pt]
\leq\ &\frac{|\pa\tl{\vp}_{\ep}|^2}{(\tl{\vp}_{\ep}+C_0)^3}+CA^2(\tl{\vp}_{\ep}+C_0)^3\frac{\tr_{\tl{\om}}\hat{\om}}{(\tr_{\hat{\om}}\tl{\om})^2}. \label{4.6}
\end{align}
If at $(t_0, z_0)$, $(\tr_{\hat{\om}}\tl{\om})^2\le A^2(\tl{\vp}_{\ep}+C_0)^3$, then 
$$H\le\log A+\frac{3}{2}\log(\tl{\vp}_{\ep}+C_0)-A\tl{\vp}_{\ep}+\frac{1}{\tl{\vp}_{\ep}+C_0}.$$ As $\tl{\vp}_{\ep}+C_0\geq 1$,  
we have an upper bound for $H$ and thus $\tl{\om}$ is bounded from above.
Otherwise, $A^2(\tl{\vp}_{\ep}+C_0)^3\le(\tr_{\hat{\om}}\tl{\om})^2$ at the maximum point. Computing the evolution of $H$, it follows from (\ref{4.5}) and (\ref{4.6}) at $(t_0, z_0)$
\begin{align*}
\lt(\p-\De_{\tl{\om}}\rt)H\le\ &C\tr_{\tl{\om}}\hat{\om}-\lt(A+\frac{1}{\tl{\vp}_{\ep}+C_0)^2}\rt)\dot{\vp}_{\ep}\notag\\[8pt]
\ &+\lt(A+\frac{1}{\tl{\vp}_{\ep}+C_0)^2}\rt)\tr_{\tl{\om}}(\tl{\om}-\hat{\om}_t+\ep_0R_h)\notag\\[8pt]
\le\ &C\tr_{\tl{\om}}\hat{\om}+(A+1)\log\frac{\Om_N}{\tl{\om}^n}+C+(A+1)n\notag\\[8pt]
 \ &-A\tr_{\tl{\om}}(\hat{\om}_t-\ep_0R_h).
\end{align*}
As $\hat{\om}_t-\ep_0R_h\ge c\hat{\om}$ and  $\lt(\p-\De_{\tl{\om}}\rt)H\le 0$ at this point, we have $$\tr_{\tl{\om}}\hat{\om}\le C\log\frac{\Om_N}{\tilde{\om}^n}+C_1.$$  for $A$ large enough. Then at $(t_0, z_0),$
\begin{align*}\tr_{\hat{\om}}\tl{\om}&\le\frac{1}{(n-1)!}(\tr_{\tl{\om}}\hat{\om})^{n-1}\frac{\det{\tl{\om}}}{\det{\hat{\om}}}\notag\\
&\le C\frac{\tl{\om}^n}{\Om_N}\lt(\log\frac{\Om_N}{\tl{\om}^n}\rt)^{n-1}+C'\notag\\
&\le \frac{C}{|s|_h^{2\beta}}
\end{align*}
as $\tilde{\om}^n/\Om_N\leq \frac{C}{|s|_h^{2\la}}$ by Lemma 4.2. Hence there exists $C>0$, independent of $\ep$, such that $H\leq C$
 for sufficiently large A. Since $\tl{\vp}_{\ep}+C_0\geq 1$, we see that $\tl{\om}$ is uniformly bounded from above. The
 lower bound follows from an argument similar to the proof of Lemma 2.3 in ~\cite{TW5}.
\end{proof}

\begin{proof}[Proof of Theorem 1.2] The existence and uniqueness is given by Proposition 2.1. The characterization of the maximal time $T_N$ follows from Theorem 1.2 in \cite{TW4}.  By Lemma 4.3, for any
 compact subset $K\subset N\setminus\{y_0\}$, there exists a positive constant $C_K$ such that $$\frac{\om_N}{C_K}\le \om(t)\le C_K \om_N\hspace*{6mm}\ \text{on}\ [T, T']\times N.$$ 
The local estimates of Gill ~\cite{Gill} then gives uniform $C^{\infty}$ estimates for $\om(t)$ on compact subsets of $N\setminus\{y_0\}$. The smooth convergence follows from this and we finish the proof. 
\end{proof}

\section{Proof of Theorem 1.3}
In this section, we prove the smoothing property for the Chern-Ricci flow with rough initial data. We follow the arguments in Song-Tian \cite{ST} closely. Using the same notations as in the introduction, assume that $\om_0$ satisfies the condition (\ref{1.3}),
then by the same arguments as in section 4 of \cite{N}, there exist functions $\psi_j\in PSH(M, \om_0)\cap C^{\infty}(M)$ such that \begin{align}\label{5.2}
\lim_{j\rightarrow \infty} \|\psi_j-\vp_0\|_{L^{\infty}(M)}=0.\end{align}
Write $\om_{0, j}=\om_0+\sdd\psi_j$. The Chern-Ricci flow starting at $\om_{0,j}$ can be 
reduced to a parabolic complex Monge-Amp\`{e}re equation. First denote
 $$T=\sup\{t\geq 0|\exists\psi\in C^{\infty}(M),\,  \ \omega_0-t\Ric(\omega_0)+\sqrt{-1}\partial\bar{\partial }\psi>0\}.$$ 
Then for any $T'<T$, there exists $\psi_{T'}\in C^{\infty}(M))$ such that $$\beta=\om_0-T'\Ric(\om_0)+\sdd\psi_{T'}>0.$$ Fix $T'<T$ 
and define smooth Hermitian metrics $$\hat{\om}_t=(1-\frac{t}{T'})\om_0+\frac{t}{T'}\beta=\om_0+t\chi$$ on $[0, T']$, 
where $\chi= \frac{1}{T'}\sdd\psi_{T'}-\Ric(\om_0).$ Let $\Omega$ be a volume form satisfying $\sdd\log\Omega=\p\hat{\om}_t=\chi.$ 
It follows that if $\vp_j$ solves the parabolic complex Monge-Amp\`{e}re equation
\begin{align}\frac{\pa\vp_j}{\pa t}=\log\frac{(\hat{\om}_t+\sdd\vp_j)^n}{\Omega},\ \ \ \ \vp_j(0)=\psi_j\end{align}
 for $t\in[0, T']$, then $\om_j=\hat{\om}_t+\sdd\vp_j$ solves the Chern-Ricci flow starting at $\om_{0,j}=\om_0+\sdd\psi_j$.  
We will show uniform $C^{\infty}$ bounds for $\vp_j$ and to prove Theorem 1.3.

Let $$
F=\frac{(\om_0+\sdd\vp_0)^n}{\Omega}\in L^p(M).$$ We use $C, C', C_i, ... $ to denote uniform constants
 depending only on $\om_0,\  ||\vp_0||_{L^{\infty}(M)}$ and $||F||_{L^p(M)}$ and varying from line to line.

 First we have the following two lemmas from ~\cite{ST} (Lemma 3.1 and 3.2). The proof is exactly the same as in ~\cite{ST}. \begin{lem}
There exists $C>0$ such that for any $t \in [0, T'],$\\
\hspace*{5.6cm}$||\varphi_j||_{L^{\infty}(M)}\leq C.$\\ 
Moreover, $\{\varphi_j\}$ is a Cauchy sequence in $C^0([0, T']\times M)$, i.e.,
$$\lim_{j,k\rightarrow\infty}||\varphi_j-\varphi_k||_{L^{\infty}([0, T']\times M)}=0.$$
\end{lem}
\begin{lem}There exists $C>0$ such that $$ \frac{t^n}{C}\leq \frac{(\hat{\omega}_t+\sqrt{-1}\partial\bar{\partial}\varphi_j)^n}{\Omega}\leq e^{\frac{C}{t}}.$$ for any $t \in [0, T'].$
\end{lem}

For convenience, we write $\om'=\om_j$ the solution of the Chern-Ricci flow starting at $\om_{0,j}$ and $g',\ \Delta',\ |\cdot|_{g'}, . . .$
 the notations corresponding to $\om_j$ for a fix $j$. Similarly, we use $\nabla^0$ to denote the covariant derivative with respect to $g_0$.  
All the bounds obtained in the following lemmas are independent of $j$. 

To prove the second order estimate,  we will need the following proposition. It follows from Proposition 3.1 in ~\cite{TW4} 
as $(T_0)_{ki\bar{l}}=(T_{0,j})_{ki\bar{l}},$ where $T_0$ and $T_{0, j}$ are the torsions corresponding to $\om_0$ and $\om_{0,j}$.
\begin{prop}{\normalfont{(Tosatti-Weinkove})} Assume that at a point $\tr_{\om_0}\om'\geq 1$, then
\begin{align*}
(\frac{\partial }{\partial t}-\Delta')\log\tr_{\omega_0}\omega'  
\leq\frac{2}{(\tr_{\omega_0}\omega')^2}\re(g'^{p\bar{q}}(T_0)^i_{pi}\nabla^{0}_{\bar{q}}\tr_{\omega_0}\omega')+C\tr_{\omega'}\omega_0
\end{align*}
at this point for some constant $C$ depending only on $g_0$.
\end{prop}

\begin{lem}There exists $C>0$ such that for $t \in(0, T]$,
$$\tr_{\omega_0}\omega' \leq e^{\frac{C}{t}}.$$
\end{lem}

\begin{proof} Let $$ H=t\log \tr _{\omega_0}\omega'+e^\Psi,$$
where $\Psi=A(\underset{[0,T']\times M}{\sup \varphi_j}-\varphi_j)$ and $A$ is a constant to be chosen later.
Assume that $H$ achieves its maximum at $(t_0, z_0)$ and $\tr_{\omega_0}\omega'>1$ (otherwise we obtain the upper
 bound for  $\tr_{\omega_0}\omega'$ directly). Choose coordinates around $(t_0, z_0)$ such that at this point,
 $(g_0)_{i\bar{j}}=\delta_{ij} $ and $(g'_{i\bar{j}}) $ is diagonal. 
First we have \begin{align*}(\frac{\partial}{\partial t}-\Delta')H=\ &t(\frac{\partial}{\partial t}-\Delta')\log \tr _{\omega_0}\omega'+\log \tr _{\omega_0}\omega' -Ae^\Psi\dot{\varphi}_j\\
\ &+Ae^\Psi\Delta'\varphi_j-A^2e^\Psi|\nabla \varphi_j|_{g'}^2. \end{align*}
It follows from Proposition 5.1 that $$(\frac{\partial }{\partial t}-\Delta')\log\tr_{\omega_0}\omega'\leq\frac{2}{(\tr_{\omega_0}\omega')^2}\re(g'^{k\bar{k}}(T_0)^i_{ki}\partial_{\bar{k}}\tr_{\omega_0}\omega')+C\tr_{\omega'}\omega_0.$$
At $(t_0, z_0),$ $\nabla_{\bar{k}}H=0$ gives $$t\frac{\partial_{\bar{k}}\tr_{\omega_0}\omega'}{\tr _{\omega_0}\omega'}-Ae^\Psi\partial_{\bar{k}}\varphi_j=0.$$ Then
 \begin{align}\label{5.3}
\begin{split}
\frac{2t}{(\tr_{\omega_0}\omega')^2}\re(g'^{k\bar{k}}(T_0)^i_{ki}\partial_{\bar{k}}\tr_{\omega_0}\omega')
&\leq \frac{2Ae^\Psi}{\tr_{\omega_0}\omega'}|\re(g'^{k\bar{k}}(T_0)^i_{ki}\partial_{\bar{k}}\varphi_j)|\\
&\leq e^\Psi(A^2|\nabla \varphi_j|_{g'}^2+C_1\tr_{\omega'}\omega_0).
\end{split}
\end{align}
Also $$\tr_{\omega_0}\omega' \leq \frac{1}{(n-1)!}(\tr_{\omega'}\omega_0)^{n-1}\frac{(\omega')^n}{\omega_0^n}.$$ Combining all the above inequalities, we get 
\begin{align}\label{5.4}
\begin{split}
&(\frac{\partial}{\partial t}-\Delta')H\\
\leq &\ C_1e^\Psi\tr_{\omega'}\omega_0+Ct\tr_{\omega'}\omega_0+\log\tr_{\omega_0}\omega'+Ane^\Psi-Ae^\Psi\tr_{\omega'}\hat{\omega_t}-Ae^\Psi\dot{\varphi}_j\\
\leq  &-e^\Psi\tr_{\omega'}(A\hat{\omega_t}-C_1\omega_0-Ct\omega_0)+(1-Ae^\Psi)\log\frac{(\omega')^n}{\omega_0^n}+(n-1)\log\tr_{\omega'}\omega_0+C_2\\
\leq &-C\tr_{\omega'}\omega_0-C_3\log t+C_4.
\end{split}
\end{align}
If we choose A to be large enough. Then \vspace{-3mm} at $(t_0, z_0),$ $$ C_5 \left (\frac{\omega_0^n}{(\omega')^n}\right )^{\frac{1}{n-1}}(\tr_{\omega_0}\omega')^{\frac{1}{n-1}}
 \leq C\tr_{\omega'}\omega_0\leq -C_3\log t+C_4.$$
So $$\log\tr_{\omega_0}\omega'\leq\log\left((\log\frac{1}{t})^{n-1}\left(\frac{(\omega')^n}{\omega_0^n}\right)\right)+C_6\leq \frac{C_7}{t}+C_8.$$
Thus H is uniformly bounded for $t \in (0, T]$ and we obtain the required estimate.
\end{proof}

Now let $S=|\nabla_{g_0}g'|_{g'}^2$. For convenience we still denote
$\Phi_{ij}^{\ \ k}=\Gamma_{ij}^{'k}-{(\Gamma_0)}_{ij}^k$, then $$S= |\Phi|^2_{g'}=g'^{i\bar{p}}g'^{j\bar{q}}g'_{k\bar{r}}\Phi_{ij}^{\ \ k}\Phi_{\bar{p}\bar{q}}^{\ \ \bar{r}}.$$

\begin{lem}
There exists $C>0 $ and $\lambda >0$ such that for $t \in(0, T']$,
$$S \leq C e^{\frac{\lambda}{t}}.$$

\end{lem}
\begin{proof}

By the evolution equation for $\tr_{\om_0}\om'$ in ~\cite{TW4} and Lemma 5.3,  we have
\begin{align*}
(\frac{\partial }{\partial t}-\Delta')\tr_{\omega_0}\omega'&\leq -C_1e^{-\frac{\alpha}{t}}S+C_2e^{\frac{\alpha}{t}}
\end{align*}
and $$(\frac{\partial }{\partial t}-\Delta')S\leq e^{\frac{\beta}{t}}(S^{3/2}+1)-\frac{1}{2}(|\overline{\nabla}'\Phi|_{g'}^2+|\nabla'\Phi|_{g'}^2).
$$ for some positive constants $\alpha$ and $\beta$. 
Take $\lambda_1>0$ such that $\frac{1}{2}e^{\frac{\lambda_1}{t}}<e^{\frac{\lambda_1}{t}}-\tr_{\omega_0}\omega'<e^{\frac{\lambda_1}{t}}$. 
Let $$H= \frac{S}{(e^{\frac{\lambda_1}{t}}-\tr_{\omega_0}\omega')^2}+e^{-\frac{\lambda_2}{t}}\tr_{\omega_0}\omega'.$$ Compute
\begingroup
\addtolength{\jot}{1em}
\begin{align*}
(\frac{\partial}{\partial t}-\Delta' ) H=&\frac{1}{(e^{\frac{\lambda_1}{t}}-\tr_{\omega_0}\omega')^2}(\frac{\partial}{\partial t}-\Delta' )S+\frac{2S}{(e^{\frac{\lambda_1}{t}}-\tr_{\omega_0}\omega')^3}(\frac{\partial}{\partial t}-\Delta' )\tr_{\omega_0}\omega'\\
&-\frac{4\re\nabla'\tr_{\omega_0}\omega' \cdot \overline{\nabla}'S}{(e^{\frac{\lambda_1}{t}}-\tr_{\omega_0}\omega')^3}-\frac{6S|\nabla' \tr_{\omega_0}\omega'|^2}{(e^{\frac{\lambda_1}{t}}-\tr_{\omega_0}\omega')^4}+\frac{\frac{2\lambda_1}{t^2}e^{\frac{\lambda_1}{t}}S}{(e^{\frac{\lambda_1}{t}}-\tr_{\omega_0}\omega')^3}\\
&+e^{-\frac{\lambda_2}{t}}(\frac{\partial}{\partial t}-\Delta' )\tr_{\omega_0}\omega'+\frac{\lambda_2}{t^2}e^{-\frac{\lambda_2}{t}}\tr_{\omega_0}\omega'\\
\leq& \ 4e^{-\frac{2\lambda_1}{t}}e^{\frac{\beta}{t}}(S^{3/2}+1)-\frac{1}{2}e^{-\frac{2\lambda_1}{t}}(|\overline{\nabla}'\Phi|_{g'}^2+|\nabla'\Phi|_{g'}^2)\\
&+\left (-2C_1e^{-\frac{\alpha}{t}}e^{-\frac{3\lambda_1}{t}}S^2+16C_2e^{\frac{\alpha}{t}}e^{-\frac{3\lambda_1}{t}}S\right)
+32e^{-\frac{3\lambda_1}{t}}|\re\nabla'\tr_{\omega_0}\omega' \cdot \overline{\nabla}'S|\\
&+C_3e^{-\frac{\lambda_1}{t}}S+\left(-C_1e^{-\frac{\alpha}{t}}e^{-\frac{\lambda_2}{t}}S+
C_2e^{\frac{\alpha}{t}}e^{-\frac{\lambda_2}{t}}\right)+C_4
\end{align*}
\endgroup
where we choose $\lambda_2$ large enough such that $2\frac{\lambda_2}{t^2}e^{-\frac{\lambda_2}{t}}\tr_{\omega_0}\omega'\leq C_4$ for some constant $C_4$.
Note that $|\nabla'\tr_{\omega_0}\omega'|_{g'}\leq \frac{1}{64}e^{\frac{\gamma}{t}}S^{1/2}$ for some $\gamma>0$ and $|\overline{\nabla}'S|_{g'}\leq 2S^{1/2}|\overline{\nabla}'\Phi|_{g'}$, we have
\begin{align*}
32e^{-\frac{3\lambda_1}{t}}|\re\nabla'\tr_{\omega_0}\omega' \cdot \overline{\nabla}'S| &\leq e^{-\frac{3\lambda_1}{t}}e^{\frac{\gamma}{t}}S|\overline{\nabla}'\Phi|_{g'}\\
&\leq \frac{1}{2}e^{-\frac{2\lambda_1}{t}}|\overline{\nabla}'\Phi|_{g'}^2+\frac{1}{2}e^{-\frac{4\lambda_1}{t}}e^{\frac{2\gamma}{t}}S^2.
\end{align*}
Also we have \vspace{.6mm}
\begingroup
\addtolength{\jot}{.6em}
\begin{align*}
4e^{-\frac{2\lambda_1}{t}}e^{\frac{\beta}{t}}S^{3/2}&\leq C_1e^{-\frac{\alpha}{t}}e^{-\frac{3\lambda_1}{t}}S^2+\frac{4}{C_1}e^{\frac{\alpha}{t}}e^{\frac{2\beta}{t}}e^{-\frac{\lambda_1}{t}}S.\end{align*}
\endgroup
Take $\lambda_2$ sufficiently large such that $e^{\frac{\alpha}{t}}e^{-\frac{\lambda_2}{t}}<1$, then fix $\lambda_2$. 
Let $\lambda_1\geq \alpha+\lambda_2$ be large enough such that $$(\frac{\partial}{\partial t}-\Delta' ) H\leq -\frac{1}{2}C_1e^{-\frac{\alpha}{t}}e^{-\frac{\lambda_2}{t}}S+C$$ for some constant $C$. Assume that $H$ achieves its maximal at $(t_0, z_0), t_0>0$, then at this point $$0\leq -\frac{1}{2}C_1e^{-\frac{\alpha}{t}}e^{-\frac{\lambda_2}{t}}S+C.$$ It follows that $H$ is bounded by some constant. 
Therefore $S\leq Ce^{\frac{\lambda}{t}}$ for some constants $C>0$ and $\lambda>0$.
\end{proof}
In addition, to bound the derivatives of $\om_j$ in the $t$-direction, it is sufficient to bound $|\Ric(g)|$ which follows from the proof of Lemma 3.4 in \cite{N}.  Then by the standard parabolic estimates~\cite{L}, 
we obtain all the higher order estimates.

\begin{prop} For any $0<\epsilon<T'$ and $k\geq 0$, there exists $ C_{\epsilon, T', k}>0$, such that 
$$  ||\varphi_j||_{C^k([\epsilon, T']\times M)}\leq C_{\epsilon, T', k}.$$
\end{prop} 
The proposition below follows from the same arguments for the proof of Proposition 3.3 in \cite{ST}. For reader's convenience, we provide the proof here.\begin{prop}
There exists a function $\varphi \in C^0([0, T)\times M)\cap C^{\infty}((0, T)\times M)$ such that  $\varphi$ is the unique solution of the equation
\begin{align} \label{5.5}
\frac{\partial{\varphi}}{\partial{t}}=\log\frac{(\hat{\omega}_t+\sqrt{-1}\partial\bar{\partial}\varphi)^n}{\Omega}, \ \ for\  t \in (0, T), \ \  \varphi|_{t=0}=\varphi_0.
\end{align}
\end{prop}

 \begin{proof} By Lemma 5.1, $\varphi_j$ is a Cauchy sequence in $C^0([0, T']\times M)$ and so we can 
define $ \varphi= \lim_{j\rightarrow \infty}\varphi_j$ which is in $C^0([0, T']\times M)$. Then it follows from 
Proposition 5.2 that for any $0<\epsilon<T'<T$, $\varphi_j$ converges to $\varphi$ in $C^{\infty}([\epsilon, T']\times M)$. 
Therefore $\varphi \in C^{\infty}((0, T)\times M)$ satisfying the above equation on $(0, T)$. Note 
that 
$$\lim_{t \rightarrow 0^+} ||\varphi(t, \cdot)-\varphi_0(\cdot)||_{L^{\infty}(M)}=0$$
 as $\vp_j(0)=\psi_j$ 
and $\psi_j\ra\vp_0$ in $L^{\infty}(M)$ as $j\ra\infty$. Then $\vp|_{t=0}=\vp_0$ and we have the existence of 
a solution $\varphi \in C^0([0, T)\times M)\cap C^{\infty}((0, T)\times M)$ for equation (\ref{5.5}).  
To prove the uniqueness, we assume that there is another solution 
$\tilde{\varphi}\in C^0([0, T)\times M)\cap C^{\infty}((0, T)\times M)$  
of equation (\ref{5.5}). Let 
$\psi=\tilde{\varphi} - \varphi$. Then $\psi$ solves the equation
$$\frac{\partial{\psi}}{\partial{t}}=\log\frac{(\hat{\omega}_t+\sqrt{-1}\partial\bar{\partial}\varphi+\sqrt{-1}\partial\bar{\partial}\psi)^n}{(\hat{\omega}_t+\sqrt{-1}\partial\bar{\partial}\varphi)^n}, \ \ for\  t \in (0, T), \ \  \psi|_{t=0}=0.$$
At any given time t, the maximum of $\psi$ is achieved at some point $z\in M$, 
then $\frac{d\psi_{max}(t)}{dt}\leq 0$ a.e. in $[0, T)$.  Similarly we have $\frac{d\psi_{min}(t)}{dt }\geq 0$ a.e. 
in $[0, T)$. As both $\psi_{max}(t)$ and $\psi_{min}(t)$ are absolutely continuous on $([0, T)$ with $\psi_{max}(0)=\psi_{min}(0)=0$, 
we have $\psi_{max}(t)\leq 0 \leq \psi_{min}(t)$  for $t\in[0, T)$. Hence $\psi(t)=0$ for $t\in[0, T)$.
\end{proof}
Now we can prove the smoothing property for the Chern-Ricci flow with rough initial data.
\begin{proof}[Proof of Theorem 1.3] If $\varphi$ is a solution of (\ref{5.5}),  then taking $\sqrt{-1}\partial\bar{\partial}$ of (\ref{5.5}) shows
that $\omega=\hat{\omega}_t+\sqrt{-1}\partial\bar{\partial}\varphi$ solves the Chern-Ricci flow on $(0, T)$ and $\lim_{t\rightarrow 0^+}\|\varphi(t,\cdot)-\varphi_0(\cdot)\|_{L^{\infty}(M)}=0$. Conversely, if  $\omega$ solves (1.1), then $$\frac{\partial}{\partial t}(\omega-\hat{\omega}_t)=\sqrt{-1}\partial\bar{\partial}\log\frac{\omega^n}{\Omega}.$$ 
Thus $\omega(t)$ must be of the form $\omega(t)=\hat{\omega}_t+\sqrt{-1}\partial\bar{\partial}\varphi$ for some $\varphi$ solving the equation 
\begin{align}\label{5.6}
\sqrt{-1}\partial\bar{\partial}(\frac{\partial\varphi}{\partial t}-\log\frac{(\hat{\omega}_t+\sqrt{-1}\partial\bar{\partial}\varphi)^n}{\Omega})=0.
\end{align}
Proposition 5.3 gives a solution of the above equation. Suppose that there exists another solution $\tilde{\varphi}\in C^{\infty}((0, T)\times M)\cap C^0([0, T)\times M)$
 of equation (\ref{5.6}). Then $$\frac{\partial\tilde{\vp}}{\partial t}=\log\frac{(\hat{\omega}_t+\sqrt{-1}\partial\bar{\partial}\tilde{\vp})^n}{\Omega}+f(t)$$ 
with $\lim_{t\rightarrow0^+}\tilde{\vp}(t)=\varphi_0$ for some smooth function $f(t)$. So we get $\varphi=\tilde{\varphi}-\int_0^tf(s)ds$ is 
a solution of the equation (\ref{5.5}) which is unique. Therefore $\tilde{\varphi}=\varphi+\int_0^tf(s)ds$ and we prove the uniqueness.
\end{proof}

\end{document}